\documentclass[abstract=true]{scrartcl}

\usepackage{lmodern}
\usepackage{amssymb}
\usepackage{amsmath}
\usepackage{amsthm}
\usepackage{xcolor}
\usepackage{bbm}
\usepackage{hyperref}
\usepackage[capitalise]{cleveref}
\usepackage{float}
\usepackage{tikz}
\usepackage{enumitem}
\usepackage{caption}
\usepackage{subcaption}

\crefname{equation}{}{}
\crefname{enumi}{}{}
\Crefname{claim}{Claim}{Claims}

\newtheorem{theorem}{Theorem}[section]
\newtheorem{lemma}[theorem]{Lemma}
\newtheorem{corollary}[theorem]{Corollary}
\newtheorem{proposition}[theorem]{Proposition}
\theoremstyle{definition}
\newtheorem{remark}[theorem]{Remark}
\newtheorem{definition}[theorem]{Definition}

\newtheorem{claim}[theorem]{Claim}
\newtheorem{example}[theorem]{Example}
\newtheorem{question}[theorem]{Question}

\numberwithin{equation}{section}

\allowdisplaybreaks
\newcommand{\R}{\mathbb{R}}

\newcommand{\Z}{\mathbb{Z}}
\newcommand{\cM}{\mathcal{M}}
\newcommand{\abs}[1]{\left|#1\right|}
\newcommand{\floor}[1]{\left\lfloor#1\right\rfloor}

\DeclareMathOperator{\conv}{\operatorname{conv}}
\DeclareMathOperator{\vol}{\operatorname{vol}}

\DeclareMathOperator{\rk}{\operatorname{rk}}
\DeclareMathOperator{\g}{g}
\DeclareMathOperator{\h}{h}
\DeclareMathOperator{\s}{s}
\newcommand{\zero}{\mathbf{0}}

\title{On generic $\Delta$-modular integer matrices with two rows}

\author{Björn Kriepke\thanks{Institute of Mathematics, University of Rostock, Germany, \href{mailto:bjoern.kriepke@uni-rostock.de}{bjoern.kriepke@uni-rostock.de}} \and Matthias Schymura\thanks{Institute of Mathematics, BTU Cottbus-Senftenberg, Germany, \href{mailto:schymura@b-tu.de}{schymura@b-tu.de}}}

\date{March 27, 2025}

\begin{document}

\maketitle

\begin{abstract}
The column number question asks for the maximal number of columns of an integer matrix with the property that all its rank size minors are bounded by a fixed parameter $\Delta$ in absolute value.
Polynomial upper bounds have been proved in various settings in recent years, with consequences for algorithmic questions in integer linear programming and matroid theory.
In this paper, we focus on the exact determination of the maximal column number of such matrices with two rows and no vanishing $2$-minors.
We prove that for large enough $\Delta$, this number is a quasi-linear function, non-decreasing and always even.
Such basic structural properties of column number functions are barely known, but expected to hold in other settings as well.
Moreover, our results identify the unique excluded (co)rank two minors for the class of matroids that are representable as a $\Delta$-submodular matrix.
\end{abstract}

\section{Introduction}

A fundamental insight in integer linear programming is that the underlying polyhedron of a linear program of the form
\[
\max\left\{ c^\intercal x : Bx \leq d , x \in \R^r \right\} \,,
\]
for $B \in \Z^{n \times r}$, $c \in \Z^r$ and $d \in \Z^n$ has only integral vertices, if $B$ is a \emph{totally unimodular} matrix, that is, every minor of any size of~$B$ equals either $-1,0$, or~$1$.
This means, that for such a program optimizing over integral solutions $x \in \Z^r$ only can be done simply by linear programming techniques, and thus in polynomial time (see, e.g.~\cite{schrijver1986theoryoflinearandinteger}).
In general however, the decision problem behind integer linear programming is $\mathrm{NP}$-hard.
In an attempt to parametrize integer linear programming with the goal to achieve efficient algorithms, the following notions have received considerable interest in the recent literature (see, e.g.~\cite{artmannweismantelzenklusen2017astrongly,
averkovschymura2023onthemaximal,
bonifasdisummaeisenbrandhaehnleniemeier2014onsubdeterminants,
fiorinijoretweltgeyuditsky2021integer,
jiangbasu2022enumerating,
oxleywalsh2022-2modular}):

For an integer matrix $A \in \Z^{r \times n}$ with full row rank $\rk(A)=r$, we denote by
\begin{equation*}
    \Delta(A) = \max\{\abs{\det B}: B \text{ is an } r \times r \text{ submatrix of }A\} \,.
\end{equation*}
We call $A$ a \emph{$\Delta$-submodular} matrix if $\Delta(A)\leq \Delta$ and call it \emph{$\Delta$-modular} if $\Delta(A)=\Delta$.\footnote{In the literature, for instance in~\cite{leepaatstallknechtxu2023polynomial}, $\Delta$-submodularity is often just called $\Delta$-modularity. It doesn't affect the statement of the results that we cite in this paper, but we prefer to make the distinction for more clarity in our arguments.}
The prevailing conjecture in the optimization community is that, for fixed~$\Delta \in \Z_{>0}$, integer linear programs with $\Delta$-modular constraint matrices can be solved in polynomial time (see~\cite{shevchenko1997qualitative}).
Motivated by this approach, there is ongoing interest in studying the combinatorics of $\Delta$-modular integer programs, and in particular, in studying the maximal size of $\Delta$-modular matrices.
We refer to~\cite{artmannweismantelzenklusen2017astrongly,kriepkeSizeIntegerPrograms2025} for thorough introductions to this topic.

In this paper, we contribute to this so-called \emph{column number problem} and determine the maximal size of generic $\Delta$-modular matrices with two rows, if~$\Delta$ is sufficiently large.
In order to make this precise and give an account of the state of the art, we say that an integer matrix~$A \in \Z^{r \times n}$ is \emph{generic}, if any~$r$ columns of~$A$ are linearly independent; we say that~$A$ is \emph{simple}, if~$A$ has only non-zero and pairwise non-parallel columns.
We always assume all matrices $A \in \Z^{r \times n}$ to have pairwise distinct columns, allowing us to consider~$A \subseteq \Z^r$ as a subset of integer vectors of cardinality~$n$.

The notion of genericity is motivated by the paper of Jiang \& Basu~\cite{jiangbasu2022enumerating}, who showed that for generic $\Delta$-modular integer programs the vertices of the integer hull of the underlying polyhedron can be enumerated in polynomial time.
The notion of simplicity is closely related to matroid theory and the growth rate question for minor-closed matroid classes (see, e.g.~\cite{geelennelsonwalsh2024excluding}).
In fact, our results have implications to excluded minor characterizations for (co)rank-$2$ $\Delta$-submodular matroids, as we discuss further below.

Now, the relevant counting functions can be formulated as
\begin{align*}
    \g(\Delta,r) &= \max\{ n: \text{there is a generic $\Delta$-modular matrix $A \in \Z^{r \times n}$ with $\rk(A)=r$}\} \\
    \s(\Delta,r) &= \max\{ n: \text{there is a simple $\Delta$-modular matrix $A \in \Z^{r \times n}$ with $\rk(A)=r$}\} \\
    \h(\Delta,r) &= \max\{ n: \text{there is a $\Delta$-modular matrix $A \in \Z^{r \times n}$ with $\rk(A)=r$}\}\, .
\end{align*}
Clearly, we have $\g(\Delta,r) \leq \s(\Delta,r) \leq \h(\Delta,r)$ for all $\Delta, r \in \Z_{>0}$.
Over recent years strong polynomial upper bounds on these functions have been obtained in terms of either of the parameters~$\Delta$ and~$r$.
The best bounds to date are
\begin{align}
\g(\Delta,r) \leq 130 \, r^3 \Delta^{\frac{2}{r}} \quad \textrm{ and } \quad
 \g(\Delta,r) \leq \begin{cases} r + 1   & \text{, if } r \geq 2\Delta - 1\\
2\Delta & \text{, if } r < 2\Delta - 1 \,, \end{cases}\label{eqn:g-bounds}
\end{align}
due to~\cite{kriepkeSizeIntegerPrograms2025};
\[
\s(\Delta,r) \leq \tbinom{r+1}{2} + 80 \Delta^7 \cdot r \quad \textrm{ for every } \quad \Delta \in \Z_{>0} \textrm{ and } r \textrm{ sufficiently large} \,,
\]
due to~\cite{paatColumnNumberForbidden2024}; and
\[
\h(\Delta,r) \leq r^2 + r + 1 + 2\,(\Delta - 1) \cdot \sum_{i=0}^4 \tbinom{r}{i} \quad \textrm{ for all } r \geq 5 \textrm{ and } \Delta \in \Z_{>0} \,,
\]
due to~\cite{averkovschymura2023onthemaximal}.
We refer to the three mentioned papers for background and references to other results in this context.

On the other hand, we know very little about basic structural properties of these counting functions, such as monotonicity or whether these functions are polynomials in~$\Delta$ or~$r$.
Also, determining these functions exactly is possible only in exceptional cases so far:
The most prominent result is due to Heller~\cite{heller1957onlinear} who showed $\h(1,r) = r^2 + r + 1$, which was the basis for many of the subsequent papers.
Lee, Paat, Stallknecht \& Xu~\cite{leepaatstallknechtxu2023polynomial} prove
\[
\h(2,r) = r^2 + 3r + 1, \textrm{ for all } r \in \Z_{>0}, \quad \textrm{ and } \quad \h(\Delta,2) = 4 \Delta + 3, \textrm{ for all } \Delta \in \Z_{>0} \,.
\]
A combination of the results in~\cite[Sect.~4]{leepaatstallknechtxu2023polynomial} and Oxley \& Walsh~\cite{oxleywalsh2022-2modular} shows that $\s(2,r)$ is completely determined, and we have
\[
\s(2,r) = \begin{cases}
\binom{r+1}{2} + r     &, \textrm{ if } r \in \{3,5\} \\
\binom{r+1}{2} + r - 1 &, \textrm{ if } r \in \Z_{>0} \setminus \{3,5\} \,.
\end{cases}   
\]
For generic matrices, the result in~\cite[Thm.~1.3]{kriepkeSizeIntegerPrograms2025} improves upon~\cite{artmanneisenbrandglanzeroertelvempalaweismantel2016anote} and shows
\[
\g(\Delta,r) = r+1 \quad \textrm{ for all } \quad r \geq 2\Delta - 1 \textrm{ and } \Delta \geq 2 \,.   
\]
Such an exact determination of $\g(\Delta,r)$ for large~$r$ was crucial in the aforementioned polynomial time enumeration algorithm of Jiang \& Basu~\cite{jiangbasu2022enumerating}.
Finally, algorithmic approaches developed in~\cite{averkovschymura2023onthemaximal} and~\cite{kriepkeSizeIntegerPrograms2025} reveal the exact values of $\g(\Delta,2)$, for $\Delta \leq 450$, of $\g(\Delta,r)$, for $r \in \{3,4,5\}$ and $\Delta \leq 40$, of $\h(\Delta,3)$, for $\Delta \leq 25$, and of $\h(\Delta,4)$, for $\Delta \leq 8$.

All these results point to a phenomenon that seems common to all the three column number functions $\g(\Delta,r)$, $\s(\Delta,r)$, $\h(\Delta,r)$:
For~$\Delta$ or~$r$ fixed, and small values of the other parameter, the function values show quite an erratic behavior and do not seem to fit into a common uniform description.
However, if~$\Delta$ or~$r$ is taken large enough with respect to the other parameter, uniform bounds or exact uniform descriptions can be given.
To the best of our knowledge, there is no good explanation for this phenomenon available to date, and more data seems to be needed in order to develop understanding.

\medskip
With this goal in mind, our starting point is our earlier paper~\cite{kriepkeSizeIntegerPrograms2025} in which we present an algorithm to compute $\g(\Delta,r)$, and where we compute the values $\g(\Delta,2)$, for every $1 \leq \Delta \leq 450$.
We observed that, for $\Delta \leq 384$, the function $\g(\Delta,2)$ shows said erratic behavior, but for $385\leq \Delta\leq 450$ it can be described uniformly.
In~\cite[Conj.~4.14]{kriepkeSizeIntegerPrograms2025} it was conjectured that this uniform description carries over asymptotically, that is, for every large enough~$\Delta$.
In this paper, we prove this conjecture:

\begin{theorem}\label{thm:upper_bound_on_g_Delta2}
For every $385 \leq \Delta \leq 1\,550$ and every $\Delta\geq 10^8$, we have
\begin{equation*}
 \g(\Delta,2) = 
 2\floor{\frac{\Delta+5}{6}} + 2\floor{\frac{\Delta+1}{3}} + 2
 =
 \begin{cases}
  \Delta + 4 & ,\textrm{ if }\Delta \equiv 2 \bmod 6 \\
  \Delta + 3 & ,\textrm{ if }\Delta \equiv 1,3,5 \bmod 6 \\
  \Delta + 2 & ,\textrm{ if }\Delta \equiv 0,4 \bmod 6 \,.
 \end{cases}
\end{equation*}
\end{theorem}

\noindent The first and finite range of values of~$\Delta$ is obtained by exhausting the computational approach from~\cite{kriepkeSizeIntegerPrograms2025}.
We iterate~\cite[Conj.~4.14]{kriepkeSizeIntegerPrograms2025} and conjecture that \cref{thm:upper_bound_on_g_Delta2} holds for all $\Delta \geq 385$ instead of~$10^8$.
The value for $\Delta=384$ equals $\g(384,2) = 387$ and does not obey the uniform description in the theorem, so that the proposed bound on~$\Delta$ would be best possible.

The gist of \cref{thm:upper_bound_on_g_Delta2} is of course the determination of~$\g(\Delta,2)$ for all $\Delta \geq 10^8$.
Our approach can be briefly described as follows:
As a variant of~$\g(\Delta,2)$, we define $\g'(\Delta,2)$ as the maximal number of columns of a generic $\Delta$-\emph{sub}modular matrix with two rows and rank two.
From the definition it is clear that $\Delta \mapsto \g'(\Delta,2)$ is non-decreasing and $\g(\Delta,2) \leq \g'(\Delta,2)$.
Furthermore, we denote by $\tilde \g(\Delta)$ the function
\begin{equation*}
    \tilde \g(\Delta) =
            \begin{cases}
                \Delta + 4 & ,\textrm{ if }\Delta \equiv 2 \bmod 6 \\
                \Delta + 3 & ,\textrm{ if }\Delta \equiv 1,3,5 \bmod 6 \\
                \Delta + 2 & ,\textrm{ if }\Delta \equiv 0,4 \bmod 6 \,.
            \end{cases}
\end{equation*}
\cref{thm:upper_bound_on_g_Delta2} then follows from the inequalities
\[
    \g(\Delta,2) \leq \g'(\Delta,2) \leq \tilde\g(\Delta) \leq \g(\Delta,2) \quad \textrm{ for every } \Delta \geq 10^8 \,.
\]
The first and last inequality (see~\cite{kriepkeSizeIntegerPrograms2025}) hold unconditionally on~$\Delta$.
The inequality in the middle is the difficult step and is based on two key observations.
First, we show in \cref{prop:omnipresence_matrices_M_a_b} that for every~$\Delta$ there is a very well-structured $\Delta$-submodular matrix $M(a,b)$ with the maximal number $\g'(\Delta,2)$ of columns.
This uses arguments from the geometry of numbers.
Secondly, in \cref{prop:upper_bound_number_columns_version_2}, we prove an upper bound on the number of columns of these matrices $M(a,b)$ in terms of the optimal value of a suitable linear program whose objective function is highly number theoretical.
Estimating the optimal value of said linear program and applying some estimates from analytical number theory then finishes the proof.
The threshold $\Delta_0 = 10^8$ in \cref{thm:upper_bound_on_g_Delta2} is imposed by the quality of both of these estimates.
Although we did not optimize the threshold as much as possible, it is clear that without significantly new ideas we will not be able to sufficiently lower the value for~$\Delta_0$ in order for the computer to compute $\g(\Delta,2)$, for the remaining $\Delta \leq \Delta_0$ (see \cref{sect:limitations}).

As a final remark on the proof of \cref{thm:upper_bound_on_g_Delta2}, we point out that the condition that we impose on the matrices $M(a,b)$ (see \cref{def:M_a_b}) is conceptually much closer to simplicity than to genericity.
In particular, it might be possible to generalize our approach for $\g(\Delta,2)=\s(\Delta,2)$ in order to determine $\s(\Delta,r)$ asymptotically for $r \geq 3$.

\medskip
\cref{thm:upper_bound_on_g_Delta2} has a number of structural consequences for the function $\g(\Delta,2)$.
We say that a property of a function $\Delta \mapsto f(\Delta)$ holds \emph{eventually}, if it holds for all $\Delta \geq \Delta_0$, for some threshold $\Delta_0 \in \Z_{>0}$.
First of all and somewhat curiously, \cref{thm:upper_bound_on_g_Delta2} shows that $\g(\Delta,2)$ is eventually even.
Note that this is not true for all~$\Delta$.
For instance, $\g(\Delta,2)$ is an odd number for $\Delta \in \{19,23,33,47,53,54,384\}$ (see~\cite[Table~2]{kriepkeSizeIntegerPrograms2025}).

Second, the first representation in \cref{thm:upper_bound_on_g_Delta2} in terms of the floor function shows that $\g(\Delta,2)$ is eventually a non-decreasing function.
The data on $\g(\Delta,r)$, for $r \in \{3,4,5\}$, obtained in~\cite{kriepkeSizeIntegerPrograms2025} suggests that this also holds for $\g(\Delta,r)$ for any fixed~$r \geq 2$.

Finally, the second representation shows that $\g(\Delta,2)$ is eventually a \emph{quasi-linear} function, more precisely, $\Delta \mapsto \g(\Delta,2)$ is linear when restricted to values of~$\Delta$ in a fixed residue class modulo~$6$.
Quasi-\emph{polynomiality} of counting functions motivated by geometric questions is a reoccurring phenomenon.
We refer to the overview paper of Woods~\cite{woods2014unreasonable} for more information on this matter.
It would be desirable to develop a structural understanding for this property, and see whether other column number functions exhibit eventual quasi-polynomiality for any of the two parameters~$\Delta$ or~$r$ as well.
In~\cite[Thm.~1.4]{kriepkeSizeIntegerPrograms2025} it is shown that for fixed $r \geq 3$, the function $\Delta \mapsto \g(\Delta,r)$ is upper bounded by a sublinear function (see~\eqref{eqn:g-bounds}), so it does not exhibit eventual (quasi-)polynomial behavior for more than two rows.
However, the following questions may be posed

\begin{question}\ 
\begin{enumerate}[label=(\roman*)]
 \item Is any of the counting functions $\g(\Delta,r),\s(\Delta,r)$, or $\h(\Delta,r)$ eventually non-decrea\-sing, for any fixed~$r$ ?
 \item Is any of the functions $\s(\Delta,r)$ or $\h(\Delta,r)$ eventually a (quasi-)polynomial, for any fixed $r \geq 3$ or any fixed~$\Delta$ ?
\end{enumerate}
\end{question}

\noindent The discussion above shows that for $\g(\Delta,2) = \s(\Delta,2)$, $\h(\Delta,2)$, $\s(2,r)$, $\h(2,r)$, and $\g(\Delta,r)$ for every fixed~$\Delta \geq 2$, both of these questions have an affirmative answer.
Monotonicity of $\s(\Delta,r)$ and $\h(\Delta,r)$ for fixed~$\Delta$ is clear, as in a (simple) $\Delta$-modular rank~$r$ matrix with the maximal number of columns, one may add an all-ones-row and then the column $(0,\ldots,0,1)^\intercal \in \Z^{r+1}$ and obtain a (simple) $\Delta$-modular matrix with full row-rank $r+1$ and one more column.
Thus, both $\s(\Delta,r)$ and $\h(\Delta,r)$ are strictly increasing as functions of~$r$.

\medskip
As we hinted at before, \cref{thm:upper_bound_on_g_Delta2} has implications on excluded minor characterizations in matroid theory.
To this end, let $\cM_\Delta$ denote the class of matroids that have a representation over~$\R$ as a $\Delta$-submodular matrix, and let $U_{r,n}$ be the rank~$r$ uniform matroid on~$n$ elements.
An excluded minor characterization of the class~$\cM_1$ is known.
A list of excluded minors for~$\cM_2$ has been identified, but it is open whether this list is complete (see~\cite[Sect.~4]{oxleywalsh2022-2modular}).
For $\Delta \geq 3$ much less is known, and an excluded minor characterization of~$\cM_\Delta$ is one of the challenging open problems in matroid theory.
Oxley \& Walsh~\cite[Prob.~4.5]{oxleywalsh2022-2modular} thus propose to first focus on rank two excluded minors for~$\cM_\Delta$.
As they discuss in~\cite[Sect.~4 \& Thm.~4.2]{oxleywalsh2022-2modular}, the uniform matroid $U_{2, \g(\Delta, 2) + 1}$ is the unique rank two excluded minor for $\cM_\Delta$, and~$\cM_\Delta$ is closed under duality.
In view of these connections, \cref{thm:upper_bound_on_g_Delta2} and the computations done in~\cite{kriepkeSizeIntegerPrograms2025} solve their problem for small enough and large enough~$\Delta$ as follows:

\begin{corollary}
\label{cor:matroids}
For every $\Delta \in \Z_{>0}$ with $\Delta \leq 1\,550$ or $\Delta\geq 10^8$, the sizes of the unique rank two excluded minor $U_{2, \g(\Delta, 2) + 1}$ for~$\cM_\Delta$ and the unique corank two excluded minor $U_{\g(\Delta, 2) - 1, \g(\Delta, 2) + 1}$ for~$\cM_\Delta$ are determined by \cref{thm:upper_bound_on_g_Delta2} and~\cite[Sect.~4.4]{kriepkeSizeIntegerPrograms2025}.
\end{corollary}

\medskip
The paper is organized as follows.
In the next section, we collect some estimates on certain number theoretic summatory functions, that we need later in the main arguments.
\cref{sect:matrices_M_a_b} focusses on the two reductions of determining $\g(\Delta,2)$ to studying certain matrices $M(a,b)$ and then to estimating the optimal value of an associated linear program to bound the number of their columns.
Moreover, \cref{subsect:small-types} discusses extremal matrices $M(a,b)$ that realize the values of $\g(\Delta,2)$ given by the function~$\tilde \g(\Delta)$.
In \cref{sect:upper_bound_linear_program}, we describe a solution of the dual problem of that linear program in order to give column number bounds for matrices $M(a,b)$ with large entries, while \cref{sect:closing-gap} then considers a computational approach to deal with those $M(a,b)$ having small entries instead.
In \cref{sect:limitations}, we discuss the limitations of our approach towards closing the gap in the valid values of~$\Delta$ in \cref{thm:upper_bound_on_g_Delta2}, and why we chose the value~$10^8$ for the threshold.

\section{Auxiliary estimates on sums involving the totient function}\label{sect:number_theory}

In this section, we recall some estimates on basic functions in number theory which we need in the rest of the paper.

For integers $k$ and $a\leq b$ we denote by $\varphi(a,b,k)$ the number of integers in the interval $[a,b]$ that are coprime to $k$, so that $\varphi(1,k,k)=\varphi(k)$ is the usual Euler totient function.
We denote by~$\vartheta(k)$ the number of square-free divisors of~$k$ and by~$\tau(k)$ the number of unitary divisors of~$k$, that is, the number of divisors $d \mid k$ with $\gcd(d,k/d)=1$.
By considering the prime factorization of~$k$, one finds a bijection between the square-free divisors and the unitary divisors of~$k$, so that $\vartheta(k)=\tau(k)$.
It is known (cf.~\cite[Lem.~3.4]{cohen1960arithmeticalunitarydivisor} for the case $a=0$) that
\begin{equation}\label{eq:asymptotic_phi_a_b_k}
    \varphi(a,b,k) = \frac{\varphi(k)}{k}(b-a) + p(a,b,k)
\end{equation}
with the error term $\abs{p(a,b,k)}\leq \tau(k)$. For the convenience of the reader we give the general proof for all $a$. Recall that 
\begin{equation*}
    \sum_{d\mid k} \mu(d) = 
        \begin{cases}
            1, & k = 1 \\
            0, & \text{else}
        \end{cases}
\end{equation*}
where $\mu$ is the Möbius function.
This implies
\begin{align}
    \varphi(a,b,k) &= \sum_{n\in[a,b]} \sum_{d\mid \gcd(n,k)} \mu(d) \notag\\
        &= \sum_{d\mid k} \mu(d)\cdot\abs{\left\{n\in[a,b]: d\mid n\right\}} \notag\\
        &= \sum_{d\mid k} \mu(d)\cdot\left(\frac{b-a}{d} + c(a,b,d)\right) \notag\\
        &= (b-a)\sum_{d\mid k} \frac{\mu(d)}{d} + \sum_{d\mid k} \mu(d)c(a,b,d) \notag\\
        &= \frac{\varphi(k)}{k}(b-a) + \sum_{d\mid k} \mu(d)c(a,b,d) \,,\label{eqn:phi-a-b-k-exact}
\end{align}
where the error terms satisfy $\abs{c(a,b,d)}\leq 1$ and therefore 
\begin{equation*}
    \abs{\sum_{d\mid k} \mu(d)c(a,b,d)} \leq \sum_{d\mid k}\abs{\mu(d)} = \vartheta(k) = \tau(k) \,.
\end{equation*}
Furthermore from~\cite{cohenNumberUnitaryDivisors1960} we have
\begin{equation*}
    \sum_{k\leq x} \tau(k) = \frac{x}{\zeta(2)}\left(\log x + 2\gamma - 1 - \frac{2\zeta'(2)}{\zeta(2)}\right) +\mathcal{O}\left( \sqrt x\log x\right)
\end{equation*}
where $\zeta$ is the Riemann $\zeta$-function, $\zeta'$ is its derivative, $\gamma$ is Euler's constant, $x$ is a real number and the sum is over all positive integers $k\leq x$. 
By inspecting the proof of this asymptotic estimate in~\cite{cohenNumberUnitaryDivisors1960}, one can derive an explicit constant and get the bound
\begin{equation}\label{eq:implied_constant_sum_tau}
    \sum_{k\leq x} \tau(k) \leq \frac{x}{\zeta(2)}\left(\log x + 2\gamma - 1 - \frac{2\zeta'(2)}{\zeta(2)}\right) + 15\sqrt x\log x \,.
\end{equation}
However, for the sake of brevity we leave the details of this to the reader.
Later we also need the following estimates for sums involving the totient function:
\begin{align}
    \sum_{k\leq x} \varphi(k) = \frac{3}{\pi^2}x^2 + p(x) \quad &\textrm{ and } \quad \sum_{k\leq x} \frac{\varphi(k)}{k} = \frac{6}{\pi^2} x + q(x) \,,\label{eqn:phi-sums-exact}
\intertext{with the error bounds}
    \abs{p(x)} \leq x(\log x+2) \quad &\textrm{ and } \quad \abs{q(x)}\leq \log x+3 \,.\notag
\end{align}
%
Asymptotic forms of these results can be found in several books on analytic number theory, for example~\cite{apostolIntroductionAnalyticNumber1986,murtyProblemsAnalyticNumber2008}. 
The first result in~\eqref{eqn:phi-sums-exact} is Proposition 2.5.3 from Jameson~\cite{jamesonPrimeNumberTheorem2003} and the second one follows by repeating essentially the same argument.
In the literature one finds better bounds on the error term, see for instance Walfisz~\cite{walfiszWeylscheExponentialsummenNeueren1963} and Liu~\cite{liuEulersFunction2016}.
However, they are much more involved and not needed for our arguments.


Using \eqref{eqn:phi-sums-exact}, we obtain the following specification:

\begin{lemma}\label{lem:explicit_value_x0}
For every $\varepsilon > 0$, there exists an $x_0(\varepsilon)$ such that for all $x \geq x_0(\varepsilon)$, we have
    \begin{align}
        (1-\varepsilon)\cdot\frac{3}{\pi^2}x^2 \leq \sum_{k\leq x} \varphi(k) \leq (1+\varepsilon)\cdot\frac{3}{\pi^2}x^2 \label{eq:bound_eps_sum_phik}
    \intertext{and}
        (1-\varepsilon)\cdot\frac{6}{\pi^2}x \leq \sum_{k\leq x} \frac{\varphi(k)}{k} \leq (1+\varepsilon)\cdot\frac{6}{\pi^2}x \,. \label{eq:bound_eps_sum_phik/k}
    \end{align}
In particular, for $\varepsilon=0.001$, we can choose $x_0(0.001) = 1880$.
\end{lemma}

\begin{proof}
In view of~\eqref{eqn:phi-sums-exact}, we just need to observe that the inequalities
\[
x(\log x+2) \leq \varepsilon \frac{3}{\pi^2}x^2
\quad\textrm{ and }\quad
\log x + 3  \leq \varepsilon \frac{6}{\pi^2}x
\]
hold for all $x \geq x_0(\varepsilon)$, for a suitably large value of $x_0(\varepsilon)$.

For $\varepsilon = 0.001$ in particular, they hold for all $x \geq 41568$.
In order to lower this value further, one may use a computer to directly check that the desired inequalities~\cref{eq:bound_eps_sum_phik,eq:bound_eps_sum_phik/k} hold for all $1880 \leq x \leq 41567$ as well.
Our code to do so can be found at \url{https://github.com/BKriepke/Generic-Delta-modular-with-two-rows}.
\end{proof}

\section{A reduction to a class of well-structured matrices}
\label{sect:matrices_M_a_b}

In this section we study a special family of matrices which we call matrices $M(a,b)$ of type $m$. These matrices have been introduced in \cite{kriepkeSizeIntegerPrograms2025} to give lower bound constructions for certain values of $\g(\Delta,2)$, see \cref{ex:familiesF1_to_F3}.
We show that every generic $\Delta$-modular matrix can, in some sense, be reduced to such a matrix $M(a,b)$. We then proceed to bound the number of columns of $M(a,b)$ by associating the matrix to a certain linear program, which we study further in \cref{sect:upper_bound_linear_program,sect:closing-gap}.

\begin{definition}\label{def:M_a_b}
    Let $m$ be a positive integer and let $a,b \in \Z^m$.
    By $M(a,b)$ we denote the matrix
    \begin{equation*}
        M(a,b) =
        \begin{pmatrix}
            0 &   1 & \cdots &   1 &   2 & \cdots &   2 &   3 & \cdots &   3 & \cdots &   m & \cdots &   m  \\
            1 & a_1 & \cdots & b_1 & a_2 & \cdots & b_2 & a_3 & \cdots & b_3 & \cdots & a_m & \cdots & b_m
        \end{pmatrix} \,,
    \end{equation*}
    where the dots in the second row represent all integers $j$ with $a_k\leq j\leq b_k$ and $\gcd(j,k)=1$.
    We call a matrix of this form a matrix of type $m$.
    We note that if $a_k > b_k$, then $M(a,b)$ has no column of the form $(k, j)^\intercal$ and if $a_k=b_k$ with $\gcd(k,a_k)=1$, then $M(a,b)$ has exactly one column of that form, namely $(k, a_k)^\intercal$.
\end{definition}

\begin{example}\label{ex:familiesF1_to_F3}
    We give three infinite families of matrices $M(a,b)$ which have the maximal number of columns of generic $\Delta$-modular matrices for infinitely many values of $\Delta$. These families have been introduced in~\cite{kriepkeSizeIntegerPrograms2025}.

    Let $m=1$ and $a=0, b=\Delta$. The matrix $M(a,b)$ given by 
    \begin{equation*}
        M(a,b) = 
            \begin{pmatrix}
                0 & 1 & 1 & 1 & \cdots & 1 \\
                1 & 0 & 1 & 2 & \cdots & \Delta 
            \end{pmatrix}
    \end{equation*}
    is generic $\Delta$-modular with $\Delta+2$ columns.

    Next, choose $m=2$ and $a=(0,\Delta)^\intercal, b=(\Delta,\Delta)^\intercal$. Then
    \begin{equation*}
        M(a,b) = 
            \begin{pmatrix}
                0 & 1 & 1 & 1 & \cdots & 1      & 2      \\
                1 & 0 & 1 & 2 & \cdots & \Delta & \Delta 
            \end{pmatrix}
    \end{equation*}
    is generic $\Delta$-modular for $\Delta$ odd and has $\Delta+3$ columns. 

    Finally, for $m=3$ we have two subfamilies:
    \begin{enumerate}[label=(\roman*)]
     \item If $\Delta = 12 s + 2$, for some $s \in \Z_{>0}$, let
     \begin{equation*}
        M(a,b) = 
        \begin{pmatrix}
            0 & 1 & \cdots &    1 &    2 & \cdots &     2 &    3 & \cdots &     3 \\
            1 & 0 & \cdots & 7s+1 & 4s+1 & \cdots & 10s+1 & 9s+1 & \cdots & 12s+2 
        \end{pmatrix}.
     \end{equation*}
     \item If $\Delta = 12 s + 8$, for some $s \in \Z_{>0}$, let
     \begin{equation*}
        M(a,b) =
        \begin{pmatrix}
            0 & 1 & \cdots &    1 &    2 & \cdots &     2 &    3 & \cdots &     3 \\
            1 & 0 & \cdots & 7s+5 & 4s+3 & \cdots & 10s+7 & 9s+7 & \cdots & 12s+8 
        \end{pmatrix}.
     \end{equation*}
    \end{enumerate}
    Both of these matrices are generic $\Delta$-modular and have $\Delta+4$ columns.

    These examples show that 
    \begin{equation*}
        \g(\Delta,2) \geq \tilde\g(\Delta) =
            \begin{cases}
                \Delta + 4 & ,\textrm{ if }\Delta \equiv 2 \bmod 6 \\
                \Delta + 3 & ,\textrm{ if }\Delta \equiv 1,3,5 \bmod 6 \\
                \Delta + 2 & ,\textrm{ if }\Delta \equiv 0,4 \bmod 6 \,.
            \end{cases}
    \end{equation*}
    as mentioned in~\cite{kriepkeSizeIntegerPrograms2025}. 
\end{example}

The reason that we can focus our study on matrices $M(a,b)$ is that roughly speaking any generic $\Delta$-modular matrix~$A$ can be turned into a matrix $M(a,b)$ of type $m$, for some number $m$, and which has at least as many columns as~$A$.
However, before we make this precise, we recall the concept of lattice-width, which allows us to give a useful bound on what the number $m$ can be.

For a non-zero vector $v \in \R^r \setminus \{\zero\}$ the \emph{width} of a convex body~$K \subseteq \R^r$ in direction~$v$ is defined as
\[
\omega(K,v) = \max_{x \in K} x^\intercal v - \min_{x \in K} x^\intercal v \,.
\]
Minimizing the width over all non-zero lattice directions yields the \emph{lattice-width}
\[
\omega_L(K) = \min_{v \in \Z^r \setminus \{\zero\}} \omega(K,v) \,.
\]
If $K$ is a \emph{lattice polytope}, meaning that $K = \conv(S)$, for a finite set $S \subseteq \Z^r$, then for $v \in \Z^r$ with $\gcd(v_1,\ldots,v_r) = 1$, the width of~$K$ in direction~$v$ can be understood as follows:
$\omega(K,v)+1$ is the number of parallel lattice-planes orthogonal to~$v$ and which have a non-empty intersection with~$K$.

\begin{lemma}[{cf.~\cite[Lemma 2.3]{kriepkeSizeIntegerPrograms2025}}]
    \label{lem:width}
    Let $A \subseteq \Z^2$ be a $\Delta$-modular point set and let $Q_A = \conv\{A \cup (-A)\}$.
    Then, the lattice-width of $Q_A$ is bounded as
    \[
    w_L(Q_A) \leq \sqrt{2 \pi} \sqrt\Delta \,.
    \]
\end{lemma}

\begin{proof}
    The proof is just a minor adjustment to the planar case of~\cite[Lemma 2.3]{kriepkeSizeIntegerPrograms2025}.
    We give the details where the argument differs:

    Let $C \subseteq Q_A$ be an area maximal quadrilateral contained in~$Q_A$.
    Since~$Q_A$ is centrally symmetric, a result of Dowker~\cite{dowker1944onminimum} shows that we may assume $C$ to be centrally symmetric as well and have its four vertices among the vertices of~$Q_A$.
    In particular, we may write $C = \conv\{\pm a_1,\pm a_2\}$, for some linearly independent $a_1,a_2 \in A$, implying $\vol(C) = 2\abs{\det(a_1,a_2)} \leq 2 \Delta$.
    Furthermore, Sas~\cite{sas1939uebereine} proved that there is always a quadrilateral contained in~$Q_A$ whose area is at least $\frac{2}{\pi}\vol(Q_A)$.

    We now combine these two results with a theorem of Makai Jr.~\cite[Theorem 3]{makai1978onthethinnest} who obtained a sharp upper bound on the lattice-width of centrally symmetric planar convex bodies, such as~$Q_A$, in terms of their volume.
    In particular, we obtain
    \[
    \omega_L(Q_A)^2 \leq 2 \vol(Q_A) \leq \pi \vol(C) \leq 2 \pi \Delta \,,
    \]
    from which the claimed inequality follows.
\end{proof}

We are now prepared to present our main reduction needed for our eventual proof of \cref{thm:upper_bound_on_g_Delta2}.

\begin{proposition}
    \label{prop:omnipresence_matrices_M_a_b}
    Let $\Delta \geq 2$ and let $A$ be a generic $\Delta$-modular matrix with two rows.
    Then, there is a positive integer $m \leq \sqrt{\pi/2} \sqrt\Delta$ and vectors $a,b\in\Z^m$ such that $M(a,b)$ is a generic $\Delta$-submodular matrix of type $m$ with at least as many columns as~$A$.
\end{proposition}

\begin{proof}
    By \cref{lem:width} we know that there is a vector $v \in \Z^2 \setminus \{\zero\}$ such that there are at most $\sqrt{2 \pi}\sqrt\Delta + 1$ parallel lines orthogonal to~$v$, so that every point of $A \cup (-A)$ is contained in one of those lines, and by symmetry, these lines are distributed symmetrically around the origin.
    By applying a suitable unimodular transformation we can assume that $v=e_1$ is the first unit vector of~$\Z^2$.
    Denote by~$B$ the set of points obtained by applying this transformation to~$A$.
    Then the first coordinates of the points in~$B$ are between $-\sqrt{\pi/2}\sqrt{\Delta}$ and $\sqrt{\pi/2}\sqrt{\Delta}$. 
    Since~$B$ is generic, for every column $b \in B$ the vector $-b$ does not occur in~$B$.
    As multiplying columns by $-1$ does not change genericity, we may assume that all columns in~$B$ have non-negative first coordinate, and by doing this we do not repeat columns by the previous observation.
    Hence, the first coordinates of~$B$ are bounded by some number $m \leq  \sqrt{\pi/2}\sqrt{\Delta}$.

    If $B$ does not include the vector $(0,1)^\intercal$, we can simply add it, as the determinants all satisfy
        \begin{equation*}
            \abs{\det\begin{pmatrix}
                0 & c \\
                1 & d
            \end{pmatrix}}
            = \abs{-c} \leq m \leq \sqrt{\frac{\pi}{2}}\sqrt\Delta \leq \Delta \,,
        \end{equation*}
    since $\Delta \geq 2$.

    Next we can divide every column $(k,j)^\intercal$ by $\gcd(k,j)$. This does not change genericity and preserves $\Delta$-submodularity. Now $B$ has the form
        \begin{equation*}
            B = 
            \begin{pmatrix}
            0 &   1 & \cdots &   1 &   2 & \cdots &   2 &   3 & \cdots &   3 & \cdots &   m & \cdots &   m  \\
            1 & a_1 & \cdots & b_1 & a_2 & \cdots & b_2 & a_3 & \cdots & b_3 & \cdots & a_m & \cdots & b_m
            \end{pmatrix} \,,
        \end{equation*}
    where the $a_k,b_k$ denote the minimum, respectively maximum, of all $j$ with $(k,j)^\intercal \in B$.
    If there is no column with first entry $k$ then we say $a_k=1$ and $b_k=0$. 

    We can now read off the vectors $a,b\in\Z^m$ and consider the matrix $M(a,b)$.
    Note that $M(a,b)$ may contain more columns than~$B$, as~$B$ does not necessarily include every column of the form $(k,j)^\intercal$ with $a_k\leq j\leq b_k$ and $\gcd(k,j)=1$ for every $k=1,\ldots,m$.
    Clearly, including these columns does not change the genericity of the matrix.
    We show that including them also preserves $\Delta$-submodularity.
    For this let $1\leq k_1,k_2\leq m$ and $j_1,j_2$ with $a_{k_i}\leq j_i\leq b_{k_i}$ for $i=1,2$.
    Then 
    \begin{align*}
        \det
        \begin{pmatrix}
            k_1 & k_2 \\
            j_1 & j_2    
        \end{pmatrix}
        &= k_1 j_2 - k_2j_1 
        \leq k_1b_{k_2}-k_2a_{k_1}
        = \det 
        \begin{pmatrix}
            k_1 & k_2 \\
            a_{k_1} & b_{k_2}
        \end{pmatrix}
        \leq \Delta 
    \intertext{and}
        \det
        \begin{pmatrix}
            k_1 & k_2 \\
            j_1 & j_2    
        \end{pmatrix}
        &= k_1 j_2 - k_2j_1 
        \geq k_1a_{k_2}-k_2b_{k_1}
        = \det 
        \begin{pmatrix}
            k_1 & k_2 \\
            b_{k_1} & a_{k_2}
        \end{pmatrix}
        \geq -\Delta
    \end{align*}
    as the two $2\times 2$ matrices on the right are submatrices of $B$ and $B$ is $\Delta$-submodular.
\end{proof}

Our next goal is to associate a linear program to approach an upper bound on the number of columns of the matrices~$M(a,b)$.
To this end, let $M(a,b)$ be a generic $\Delta$-modular matrix of type $m$ and let $1\leq k,\ell\leq m$ with $b_k-a_k, b_\ell-a_\ell\geq 0$. Using the $\Delta$-modularity of $M(a,b)$, we obtain
\begin{equation*}
  \ell (b_k - a_k) + k (b_\ell - a_\ell)
  = \det\begin{pmatrix}
    \ell & k \\
    a_\ell & b_k
  \end{pmatrix}
  +
  \det\begin{pmatrix}
    k & \ell \\
    a_k & b_\ell
  \end{pmatrix}
  \leq 2 \Delta \,.
\end{equation*}
We set $x_k = \frac{b_k-a_k}{\Delta k}\geq 0$.
From the previous inequality we obtain 
\begin{equation*}
    x_k + x_\ell \leq \frac{2}{k\ell} \,.
\end{equation*}
We consider the following linear program:
\begin{align}
  \sum_{k=1}^m \varphi(k) x_k &\to \max   \notag \\
  x_k + x_\ell &\leq \frac{2}{k\ell} \qquad k, \ell=1, \ldots, m \label{lp:primal_version}\\
   x &\geq 0 \,. \notag
\end{align}
%


\begin{lemma}
\label{lem:opt_value_exists}
    The optimal value of~\cref{lp:primal_version} exists and is upper bounded by $\sum_{k=1}^m \frac{\varphi(k)}{k^2}$.
\end{lemma}

\begin{proof}
    Let $x$ be any feasible solution of~\cref{lp:primal_version}.
    Setting $k=\ell$ in the constraint, we obtain that $x_k \leq 1/k^2$, for all $k=1,\ldots,m$, which together with the constraint $x \geq 0$ implies that the feasible set is compact.
    Hence, the optimal value of~\cref{lp:primal_version} exists and we have 
    \begin{equation*}
        \sum_{k=1}^m \varphi(k) x_k \leq \sum_{k=1}^m \frac{\varphi(k)}{k^2} \,,
    \end{equation*}
    as claimed.
\end{proof}

\begin{remark}
Although we could not make use of this, we mention in passing, that the polytope underlying the linear program~\eqref{lp:primal_version} is a dilate of an alcoved polytope of type~$C$.
These polytopes associated to root systems have a very rich combinatorial structure (see~\cite{lampostnikov2018alcovedpolytopesII}).
\end{remark}

It is known (see for example~\cite[Lemma 2.4]{cohen1960arithmeticaldivisor}) that
\begin{equation}\label{eq:asymptotic_phi_div_k^2}
    \sum_{k\leq x} \frac{\varphi(k)}{k^2} = \frac{1}{\zeta(2)}\left(\log x  - \frac{\zeta'(2)}{\zeta(2)}+\gamma \right) + \mathcal{O}\left(\frac{\log x}{x}\right) \,.
\end{equation}
In particular, the upper bound given in \cref{lem:opt_value_exists} tends to infinity for $m\to\infty$.
However, we later show that the optimal value of~\eqref{lp:primal_version} is instead upper bounded by~$1$ for all integers~$m$.
The motivation for showing this is the following key proposition.

\begin{proposition}
\label{prop:upper_bound_number_columns_version_2}
    Let $m$ be a positive integer and let $a,b\in\Z^m$ give a generic $\Delta$-submodular matrix $M(a,b)$ of type $m$.
    Let $z_m\in\R$ be the optimal value of~\cref{lp:primal_version} (which exists by \cref{lem:opt_value_exists}).
    Then the number of columns of $M(a,b)$ is bounded as
    \begin{equation*}
        \abs{M(a,b)}\leq z_m\Delta + \mathcal{O}(m\log m) \,.
    \end{equation*}
\end{proposition}

\begin{proof}
    We start by defining two sets
    \begin{align*}
        P(a,b) &= \{k\in[m]: b_k-a_k\geq 0 \} \\
        N(a,b) &= \{k\in[m]: b_k-a_k < 0 \} \,,
    \end{align*}
    in other words, $k\in P(a,b)$ if $M(a,b)$ contains at least one column of the form $(k,j)^\intercal$ and $k\in N(a,b)$ otherwise. It follows that
    \begin{equation*}
        \abs{M(a,b)} = 1 + \sum_{k\in P(a,b)} \varphi(a_k,b_k,k)
    \end{equation*}
    where the $1$ at the beginning of the sum counts the column $(0,1)^\intercal$. 

    By the considerations above, we know that setting $x_k = (b_k-a_k)/(\Delta k)$ for $k\in P(a,b)$ and $x_k=0$ otherwise gives a feasible solution to the linear program \cref{lp:primal_version}.
    In particular, letting $x^*\in\R^m$ with entries $x_k^*$ be an optimal solution of \cref{lp:primal_version}, then 
    \begin{equation*}
        \sum_{k\in P(a,b)}\varphi(k)\frac{b_k-a_k}{k} = \Delta \sum_{k=1}^m \varphi(k)x_k \leq \Delta \sum_{k=1}^m \varphi(k) x_k^* = \Delta z_m \,.
    \end{equation*}
    In conclusion we obtain 
    \begin{align}
        \abs{M(a,b)} &= 1 + \sum_{k=1}^m \varphi(a_k,b_k,k) \notag \\
            &\leq 1+ \sum_{k\in P(a,b)} \left(\varphi(k) \frac{b_k-a_k}{k} + \tau(k)\right) \label{eq:use_error_bound_phi(a_b_k)}\\
            &\leq \sum_{k\in P(a,b)} \varphi(k)\frac{b_k-a_k}{k} + \left(1 + \sum_{k=1}^m \tau(k)\right) \notag \\
            &\leq z_m\Delta + 1+ \frac{m}{\zeta(2)}\left(\log m + 2\gamma - 1 - \frac{2\zeta'(2)}{\zeta(2)}\right) + 15\sqrt m\log m \label{eq:implied_constant_key_proposition} \\
            &= z_m\Delta + \mathcal{O}(m\log m) \notag 
    \end{align}
    where we used \cref{eq:asymptotic_phi_a_b_k,eq:implied_constant_sum_tau}.
\end{proof}

\subsection{The maximal size of the matrices \texorpdfstring{$M(a,b)$}{M(a,b)} for small \texorpdfstring{$m$}{m}}
\label{subsect:small-types}

For small values of~$k$ the asymptotic~\eqref{eq:asymptotic_phi_a_b_k} can be replaced by the identity~\eqref{eqn:phi-a-b-k-exact}.
Observe that in a matrix $M(a,b)$ we can always assume $a_k,b_k \not\equiv 0 \bmod k$, for every $2 \leq k \leq m$.
With this assumption we find that
\[
    c(a_1,b_1,1) = 1 \quad\textrm{ and }\quad c(a_2,b_2,2) = 0 \,,
\]
and $c(a_3,b_3,3)$ only depends on the values of $a_3,b_3 \bmod 3$ and can quickly be found by considering the few cases:
\begin{equation}\label{eq:values_c_a_b}
    c(a_3,b_3,3) = \begin{cases}
        \phantom{-}\frac13   &\textrm{, if } (a_3,b_3) \equiv (2,1) \bmod 3 \\
        \phantom{-}0         &\textrm{, if } (a_3,b_3) \equiv (1,1),(2,2) \bmod 3 \\
        -\frac13  &\textrm{, if } (a_3,b_3) \equiv (1,2) \bmod 3 \,.
    \end{cases}
\end{equation}
In general, we have $\abs{c(a,b,d)} \leq \frac{d-2}{d}$ if $a,b \not\equiv 0 \bmod d$.
Therefore, we obtain
\begin{align}
    \varphi(a_1,b_1,1) &= b_1 - a_1 + 1 \notag\\
    \varphi(a_2,b_2,2) &= \frac{1}{2}(b_2-a_2) + c(a_2,b_2,1) - c(a_2,b_2,2) = \frac{1}{2}(b_2-a_2) + 1 \label{eqn:three-varphis}\\
    \varphi(a_3,b_3,3) &= \frac{2}{3}(b_3-a_3) + c(a_3,b_3,1) - c(a_3,b_3,3) \leq \frac{2}{3}(b_3-a_3) + \frac43 \,. \notag
\end{align}
Furthermore the optimal value of \cref{lp:primal_version} can be computed for small $m$, giving $z_m=1$ for $m=1,2,3$. Repeating the proof of \cref{prop:upper_bound_number_columns_version_2} in these cases gives:

\begin{corollary}
    \label{cor:upper_bound_m_1_2_3}
    Let $m \in \{1,2,3\}$ and let $a,b \in \Z^m$ give a generic $\Delta$-submodular matrix $M(a,b)$ of type $m$.
    Then
    \begin{equation*}
        \abs{M(a,b)} \leq 
            \begin{cases}
                \Delta + 2 & , \ m=1 \\
                \Delta + 3 & , \ m=2 \\
                \Delta + 4 & , \ m=3 \,.
            \end{cases}
    \end{equation*}
\end{corollary}

\begin{proof}
    The proof is the same as for \cref{prop:upper_bound_number_columns_version_2} with only minor adjustments.
    In particular, due to~\eqref{eqn:three-varphis}, inequality~\cref{eq:use_error_bound_phi(a_b_k)} now reads
    \begin{align*}
        \abs{M(a,b)} \leq 
            \begin{cases}
                1 + c_1\left(b_1-a_1 + 1\right) &,\ m=1 \\
                1 + c_1\left(b_1-a_1 + 1\right) + c_2\left(\frac{b_2-a_2}{2} + 1\right) &,\ m=2 \\
                1 + c_1\left(b_1-a_1 + 1\right) + c_2\left(\frac{b_2-a_2}{2} + 1\right) + c_3\left(\frac23(b_3-a_3) + \frac43\right) &,\ m=3
            \end{cases}
    \end{align*}
    where $c_k=1$ if $k\in P(a,b)$ and $c_k=0$ otherwise for $k=1,2,3$. The result follows with $z_m=1$ for $m=1,2,3$ and using that $\abs{M(a,b)}$ is an integer.
\end{proof}

\subsubsection{Matrices \texorpdfstring{$M(a,b)$}{M(a,b)} of type \texorpdfstring{2}{2}}

In this section, we improve \cref{cor:upper_bound_m_1_2_3} in the case $m=2$.
We have seen in \cref{ex:familiesF1_to_F3} that for odd $\Delta$ there is a generic $\Delta$-modular matrix of type $2$ with $\Delta+3$ columns, meeting the upper bound of \cref{cor:upper_bound_m_1_2_3}. Here we want to show that for even $\Delta$ this bound is not sharp and we can achieve at most $\Delta+2$ columns.

For the sake of contradiction, let $\Delta$ be even and assume that there is a generic $\Delta$-submodular matrix $A=M(a,b)$ of type~$2$ with $\Delta+3$ columns.
Then
\begin{equation*}
    A = \begin{pmatrix}
        0 &   1 & \cdots &   1 &   2 & \cdots &   2 \\
        1 & a_1 & \cdots & b_1 & a_2 & \cdots & b_2
    \end{pmatrix} \in \Z^{2\times(\Delta+3)} \,.
\end{equation*}
By subtracting the first row from the second row we can assume that $a_1=0$.
Furthermore $a_2\leq b_2$ and both $a_2,b_2$ need to be odd.
The matrix~$A$ has 
\begin{equation*}
    1 + (b_1+1) + \frac{b_2-a_2}{2}+1 = b_1 + \frac{b_2-a_2}{2} + 3 \overset{!}{=} \Delta + 3
\end{equation*}
columns, hence $2b_1 + b_2-a_2= 2\Delta$.
The condition that $A$ is $\Delta$-submodular implies that 
\[
    \det\begin{pmatrix}
    1 &   2 \\
    0 & b_2 
    \end{pmatrix}
    = b_2 \leq \Delta
    \quad\textrm{ and }\quad
    \det\begin{pmatrix}
    1 &   2 \\
    b_1 & a_2 
    \end{pmatrix}
    = a_2 - 2b_1 \geq -\Delta \,.
\]
Adding $2\Delta=2b_1+b_2-a_2$ to the second inequality, we obtain 
\begin{equation*}
\Delta \leq a_2-2b_1 + 2\Delta = a_2 - 2b_1 + 2b_1 + b_2 - a_2 = b_2 \,,
\end{equation*}
and combining this with the first inequality we get $b_2=\Delta$. However, $b_2$ is odd while~$\Delta$ is even, leading to a contradiction. We have therefore shown the following:

\begin{claim}\label{claim:matrices_type_2}
Let $a,b\in\Z^2$ give a generic $\Delta$-submodular matrix $M(a,b)$ of type $2$. Then
    \begin{equation*}
        \abs{M(a,b)}\leq 
            \begin{cases}
                \Delta + 3 & ,\textrm{ if }\Delta \text{ is odd} \\
                \Delta + 2 & ,\textrm{ if }\Delta \text{ is even} \,.
            \end{cases}
    \end{equation*}
\end{claim}
We remark that this bound is sharp, because if $\Delta$ is even, then $\Delta'=\Delta-1$ is odd, meaning we can use the construction from \cref{ex:familiesF1_to_F3} giving a generic $\Delta'$-modular (and therefore also generic $\Delta$-submodular) matrix
\begin{equation*}
    \begin{pmatrix}
        0 & 1 & \cdots &      1 &      2 \\
        1 & 0 & \cdots & \Delta' & \Delta'
    \end{pmatrix}
    =
    \begin{pmatrix}
        0 & 1 & \cdots &        1 &        2 \\
        1 & 0 & \cdots & \Delta-1 & \Delta-1
    \end{pmatrix}
\end{equation*}
with $\Delta'+3=(\Delta-1)+3=\Delta+2$ columns.
Furthermore, in the case that $\Delta=4n$
\begin{equation*}
    \begin{pmatrix}
        0 & 1 & \cdots &    1 &    2 & \cdots &    2 \\
        1 & 0 & \cdots & 3n-1 & 2n-1 & \cdots & 4n-1
    \end{pmatrix}
    =
    \begin{pmatrix}
        0 & 1 & \cdots &               1 &               2 & \cdots &        2 \\
        1 & 0 & \cdots & \frac34\Delta-1 & \frac12\Delta-1 & \cdots & \Delta-1
    \end{pmatrix}
\end{equation*}
is a generic $\Delta$-modular matrix with $\Delta+2$ columns.
In fact it can be shown that this is the only such matrix of type~$2$; however, this is not relevant for our main results, so we leave out the details.

\subsubsection{Matrices \texorpdfstring{$M(a,b)$}{M(a,b)} of type \texorpdfstring{3}{3}}

Here we prove the analogous result to \cref{claim:matrices_type_2} for matrices $M(a,b)$ of type $3$:

\begin{claim}\label{claim:matrices_type_3}
    Let $a,b\in\Z^3$ give a generic $\Delta$-submodular matrix of type $3$. Then
    \begin{equation*}
        \abs{M(a,b)}\leq 
            \begin{cases}
                \Delta + 4 & ,\textrm{ if }\Delta \equiv 2 \bmod 6 \\
                \Delta + 3 & ,\textrm{ if }\Delta \equiv 1,3,5 \bmod 6 \\
                \Delta + 2 & ,\textrm{ if }\Delta \equiv 0,4 \bmod 6 \,.
            \end{cases}
    \end{equation*}
\end{claim}
As by \cref{cor:upper_bound_m_1_2_3} we always have $\abs{M(a,b)}\leq \Delta+4$, it suffices to show that for $\Delta\equiv 1,3,5 \bmod 6$ we cannot obtain $\Delta+4$ columns and that for $\Delta\equiv 0,4\bmod 6$ we cannot get $\Delta+3$ or $\Delta+4$ columns.
With this in mind, let $\Delta \not\equiv 2 \bmod 6$ and 
\begin{equation*}
    A = M(a,b) = \begin{pmatrix}
        0 & 1 & \cdots &   1 &   2 & \cdots &   2 &   3 & \cdots &   3 \\
        1 & a_1 & \cdots & b_1 & a_2 & \cdots & b_2 & a_3 & \cdots & b_3 
    \end{pmatrix}
\end{equation*}
be a generic $\Delta$-submodular matrix, where we can again assume that $a_1 = 0$.
In view of~\eqref{eqn:three-varphis}, the number of columns of $A$ is given by
\begin{equation*}
    1 + (b_1-a_1+1) + \left(\frac{b_2-a_2}{2}+1\right) + \left(\frac{2}{3}(b_3-a_3) + 1 - c(a_3,b_3,3)\right) \overset{!}{=} \Delta + d \,,
\end{equation*}
where the values of $c(a_3,b_3,3)$ are given in \cref{eq:values_c_a_b} and $d\in\{3,4\}$.
We want to show that if $\Delta$ is odd, then $d=4$ is impossible, and if $\Delta \equiv 0,4 \bmod 6$ then $d=3,4$ are both impossible.
First, subtract $4$ on both sides and multiply by $6$ to get 
\begin{equation}\label{eq:equation_for_6Delta}
    6b_1 + 3(b_2-a_2) + 4(b_3-a_3) - 6c(a_3,b_3,3) = 6\Delta + 6(d-4) \,.
\end{equation}
As $A$ is $\Delta$-submodular, we have
\[
    \det\begin{pmatrix}
    1 &   3 \\
    0 & b_3 
    \end{pmatrix}
    = b_3 \leq \Delta
    \quad,\quad
    \det\begin{pmatrix}
    3 &   1 \\
    a_3 & b_1  
    \end{pmatrix}
    = 3b_1-a_3 \leq \Delta
\]
and
\[
    \det\begin{pmatrix}
    2 &   3 \\
    a_2 & b_3 
    \end{pmatrix}
    = 2b_3-3a_2 \leq \Delta
    \quad,\quad
    \det\begin{pmatrix}
    3 &   2 \\
    a_3 & b_2 
    \end{pmatrix}
    = 3b_2-2a_3 \leq \Delta \,. 
\]
Now, introduce integer variables $e_i\geq 0$ for $i=1,2,3,4$ so that
\begin{align*}
    b_3 + e_1 &= \Delta \\
    3b_1 - a_3 + e_2 &= \Delta \\
    2b_3 - 3a_2 + e_3 &= \Delta \\
    3b_2 - 2a_3 + e_4 &= \Delta \,.
\end{align*}
Multiplying the first two equations by $2$ and then adding up all equations leads to 
\begin{equation*}
    6b_1 + 3(b_2-a_2) + 4(b_3-a_3) + 2e_1+2e_2+e_3+e_4 = 6\Delta
\end{equation*}
so that combining with \cref{eq:equation_for_6Delta} gives
\begin{equation}\label{eq:equation_for_e_i}
    2e_1 + 2e_2 + e_3 + e_4 = 6 \left((4 - d) - c(a_3,b_3,3)\right) \,.
\end{equation}
\cref{claim:matrices_type_3} is now implied by the following claim:

\begin{claim}
    There are no integer solutions to the system of equations
    \begin{align*}
        b_3 + e_1 &= \Delta\\
        3b_1 - a_3 + e_2 &= \Delta \\
        2b_3 - 3a_2 + e_3 &= \Delta \\
        3b_2 - 2a_3 + e_4 &= \Delta \\
        2e_1 + 2e_2 + e_3 + e_4 &= 6 \left((4 - d) - c(a_3,b_3,3)\right) \,,
    \end{align*}
    with $a_2,b_2$ odd, $a_3,b_3$ not divisible by $3$, $e_i\geq 0$ and 
    \begin{itemize}
        \item $\Delta \equiv 0,4 \bmod 6$ and $d=3,4$, or
        \item $\Delta \equiv 1,3,5 \bmod 6$ and $d=4$.
    \end{itemize}
\end{claim}

\begin{proof}
    First of all, consider the case $d=4$.
    Then, the last identity reduces to $2e_1 + 2e_2 + e_3 + e_4 = -6 c(a_3,b_3,3)$, so that by~\eqref{eq:values_c_a_b}, we must have $c(a_3,b_3,3) \in \{0,-1/3\}$.

    \emph{Case 1:} $c(a_3,b_3,3) = 0$.

    \noindent This implies that $e_1 = e_2 = e_3 = e_4 = 0$, and thus $\Delta = b_3 = 2b_3 - 3a_2$, in particular $\Delta \equiv b_3 \equiv 2b_3 \bmod 3$, which contradicts the assumption that~$b_3$ is not divisible by~$3$.

    \emph{Case 2:} $c(a_3,b_3,3) = -1/3$.

    \noindent This means that $2e_1+2e_2+e_3+e_4=2$ and by~\eqref{eq:values_c_a_b}, we have $(a_3,b_3) \equiv (1,2) \bmod 3$.
    Since the $e_i$ are all non-negative integers, the only options are $e = (e_1,e_2,e_3,e_4) \in \{(1,0,0,0),(0,1,0,0),(0,0,1,1),(0,0,2,0),(0,0,0,2)\}$, which we handle separately as follows.

    \emph{Case 2.1:} $e = (1,0,0,0)$ and thus $\Delta = b_3 + 1 = 3b_1 - a_3$.

    \noindent We get $\Delta \equiv b_3+1 \equiv 0 \equiv -a_3 \bmod 3$, in contradiction to~$a_3$ not being divisible by~$3$.

    \emph{Case 2.2:} $e = (0,1,0,0)$ and thus $\Delta = b_3 = 3b_1 - a_3 + 1$.

    \noindent We get $\Delta \equiv b_3 \equiv 2 \equiv 1-a_3 \bmod 3$, in contradiction to $a_3 \equiv 1 \bmod 3$.

    \emph{Case 2.3:} $e = (0,0,1,1)$ and thus $\Delta = b_3 = 2b_3 - 3a_2 + 1$.

    \noindent Writing $b_3 = 3k + 2$ and $a_2=2\ell+1$, for some $k,\ell \in \Z$, we get $\Delta \equiv b_3 \equiv 2(3k+2)-3(2\ell+1)+1 \equiv 2 \bmod 6$, in contradiction to $\Delta \not\equiv 2 \bmod 6$.

    \emph{Case 2.4:} $e = (0,0,2,0)$ and thus $\Delta = b_3 = 2b_3 - 3a_2 + 2$.

    \noindent We get $\Delta \equiv b_3 \equiv 2b_3+2 \equiv 0 \bmod 3$, in contradiction to~$b_3$ not being divisible by~$3$.

    \emph{Case 2.5:} $e = (0,0,0,2)$ and thus $\Delta = b_3 = 3b_2 - 2a_3 + 2$.

    \noindent We get $\Delta \equiv b_3 \equiv -2a_3+2 \equiv 0 \bmod 3$, in contradiction to~$b_3$ not being divisible by~$3$.

    \smallskip

    It remains to consider the case $d=3$.
    Here, the last identity in the linear system reduces to $2e_1 + 2e_2 + e_3 + e_4 = 6 (1-c(a_3,b_3,3))$, so that by~\eqref{eq:values_c_a_b}, all three possibilities $c(a_3,b_3,3) \in \{0,-1/3,1/3\}$ may happen, leading to the cases where $2e_1 + 2e_2 + e_3 + e_4 \in \{4,6,8\}$.
    One may argue similarly as for $d=4$, but there will be considerably more cases to check.
    This is a finite computation modulo~$6$, and we instead refer to a small computer program, available at \url{https://github.com/BKriepke/Generic-Delta-modular-with-two-rows}.
\end{proof}

We note that for $d=4$ we have only used the condition $\Delta\not\equiv 2\bmod 6$ in Case 2.3. Therefore all examples of generic $\Delta$-modular matrices of type $3$ with $\Delta+4$ columns need to belong to that case. A short analysis then gives the class
\begin{equation*}
    M(a,b) = \begin{pmatrix}
        0 & 1 & \cdots &   1 &   2 & \cdots &   2 &   3 & \cdots &   3 \\
        1 & 0 & \cdots & \frac{\Delta+a_3}{3} & \frac{\Delta+1}{3} & \cdots & \frac{\Delta+2a_3-1}{3} & a_3 & \cdots & \Delta 
    \end{pmatrix}
\end{equation*}
where $\frac23\Delta \leq a_3 \leq \frac34 \Delta+1$ and $a_3\equiv 1\bmod 3$.
Note that for the families in \cref{ex:familiesF1_to_F3} we had 
\begin{equation*}
    a_3 = \begin{cases}
        9s+1 = \frac34\Delta-\frac12 &, \text{ if } \Delta = 12s+2 \\
        9s+7 = \frac34\Delta+1       &, \text{ if } \Delta = 12s+8\,.
    \end{cases}
\end{equation*}
We do not know if matrices of the above form are inequivalent (in the sense of Definition~4.5 of~\cite{kriepkeSizeIntegerPrograms2025}) for different values of $a_3$. 

Combining \cref{cor:upper_bound_m_1_2_3,claim:matrices_type_2,claim:matrices_type_3} now gives:

\begin{corollary}
\label{cor:precise_upper_bound_for_m_1_2_3}
    Let $m\in\{1,2,3\}$ and let $a,b\in\Z^m$ give a generic $\Delta$-submodular matrix $M(a,b)$ of type $m$. Then
    \begin{equation*}
        \abs{M(a,b)} \leq 
        \begin{cases}
            \Delta + 4 & ,\textrm{ if }\Delta \equiv 2 \bmod 6 \\
            \Delta + 3 & ,\textrm{ if }\Delta \equiv 1,3,5 \bmod 6 \\
            \Delta + 2 & ,\textrm{ if }\Delta \equiv 0,4 \bmod 6 \,.
        \end{cases}
    \end{equation*}
\end{corollary}

\section{An upper bound on the value of \texorpdfstring{\eqref{lp:primal_version}}{(\ref{lp:primal_version})} for large dimensions}\label{sect:upper_bound_linear_program}

Our goal from here on is to show that $z_m \leq w$ for all $m>3$ for a suitable choice of $w<1$ independent of $m$.
The argument is split up into two sections.
First, we consider the dual problem of~\cref{lp:primal_version} and find a feasible solution to this dual with objective value bounded by~$w=0.999$, for $m \geq 3257$.
In \cref{sect:closing-gap}, we then close the gap by a computational approach for the remaining dimensions.

Remember that for a given linear program in the form
\begin{align*}
    c^Tx &\to \max \\*
    Ax &\leq b \\*
    x &\geq 0
\end{align*}
the dual problem is given by
\begin{align*}
    b^T y &\to \min \\*
    A^Ty &\geq c \\*
    y &\geq 0 \,.
\end{align*}
We refer to~\cite{schrijver1986theoryoflinearandinteger} for an introduction to linear programming and its duality theory.

In our concrete case of \cref{lp:primal_version}, the vector $c\in\R^m$ has coordinates $c_k=\varphi(k)$, $k \in [m]$.
Instead of thinking about $b\in\R^{m^2}$ as a vector, for indexing purposes we want to think about it as a matrix $B = (B_{k,\ell}) \in \R^{m\times m}$ with $B_{k,\ell} = 2/(k\ell)$.
The matrix $A\in \R^{m^2 \times m}$ is given by  
\begin{equation*}
    A = 
    \begin{pmatrix}
        1 \\
        \vdots \\
        1 \\
        & 1 \\
        & \vdots \\
        & 1 \\
        &   & \ddots \\
        &   &        & 1 \\
        &   &        & \vdots \\
        &   &        & 1
    \end{pmatrix}
    +
    \begin{pmatrix}
        1 \\
        & \ddots \\
        &        & 1 \\
        1 \\
        & \ddots \\
        &        & 1 \\
        & \vdots \\
        1 \\
        & \ddots \\
        &        & 1
    \end{pmatrix} \,.
\end{equation*}
Again we want to think of the variable vector $y\in \R^{m^2}$ as a matrix $Y = (Y_{k,\ell}) \in \R^{m\times m}$.
Then, the dual problem can be written as 
\begin{align}
    \sum_{k,\ell=1}^m \frac{2}{k\ell} Y_{k,\ell}&\to \min \notag \\*
    \sum_{\ell=1}^m (Y_{k,\ell} + Y_{\ell,k}) &\geq \varphi(k), \qquad k=1,\ldots, m,\tag{$D_m$}\label{lp:dual_version}\\*
    Y&\geq 0. \notag 
\end{align}
We now want to find a feasible solution to \cref{lp:dual_version} with objective value $z_m^* \leq w$.
By weak duality we have $z_m \leq z_m^*$ and therefore $z_m \leq w$, as we wanted.

Before proceeding, we motivate our choice for a suitable feasible solution. For small values of $m$ the LP \cref{lp:dual_version} can be easily solved with a computer and the solution can be inspected. An optimal solution for $m=100$ is shown in \cref{fig:shape-solutions-Y}~(a). It holds that $Y_{k,\ell}=0$ for most $k,\ell$ and $Y_{k,\ell}\neq 0$ only for pairs $(k,\ell)$ with $k^2+\ell^2\approx m^2$, giving a picture of a quarter circle.
%
%
However, it seems quite hard to give a closed form for the entries $Y_{k,\ell}$ of an optimal solution~$Y$ of \cref{lp:dual_version} for general~$m$.
Instead we want to approximate the circle by enclosing it in three rectangles.
An illustration of this idea is given in \cref{fig:shape-solutions-Y}~(b).


\begin{figure}
\centering
 \subfloat[\centering Optimal solution for $m=100$. The dots represent the non-zero entries in~$Y$.]{{\includegraphics[width=.45\textwidth]{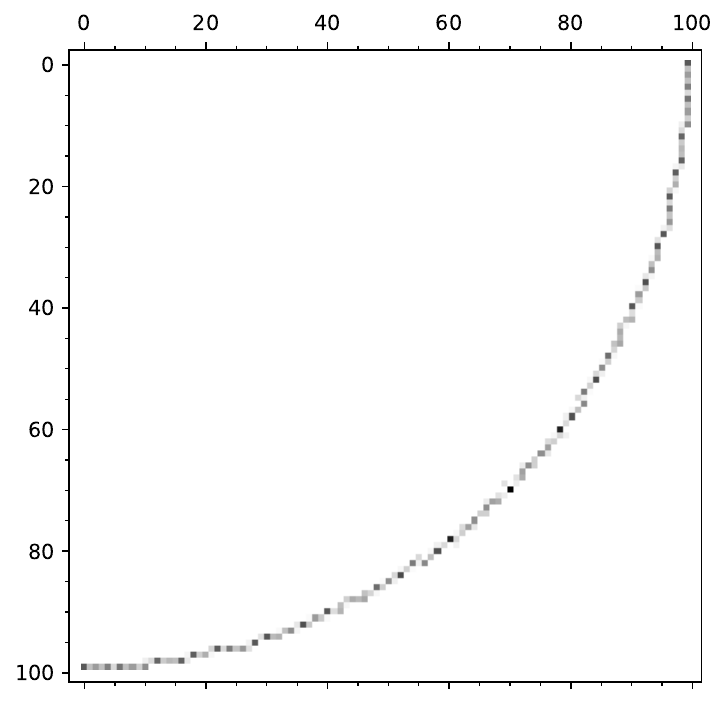} }}%
  \label{fig:solution_m_100}
 \qquad
 \subfloat[\centering Abstraction to a general solution~$Y$.]{{\includegraphics[width=.41\textwidth]{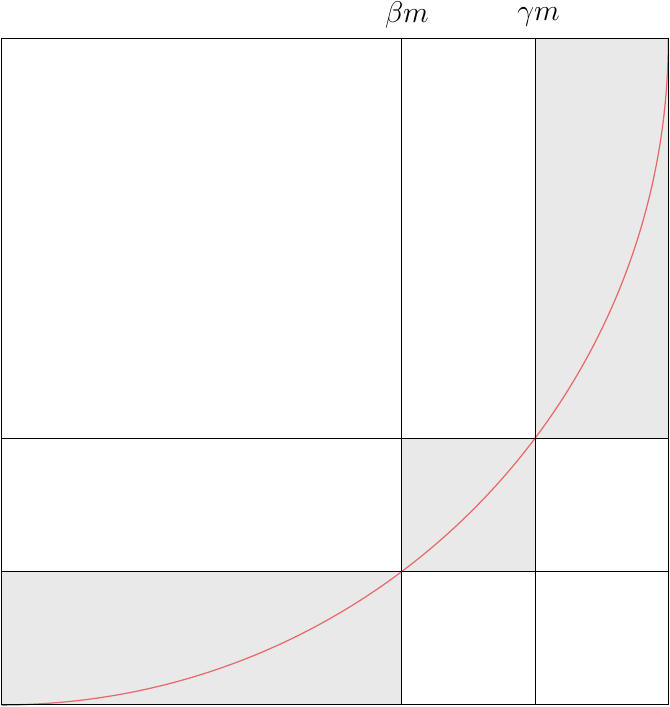} }}%
  \label{fig:solution_Y}
\caption{The shape of (near) optimal solutions~$Y$ to the dual problem \cref{lp:dual_version}.}
\label{fig:shape-solutions-Y}
\end{figure}

We choose the constants $\beta=\sqrt{1/3}$ and $\gamma=\sqrt{2/3}$ so that $(\beta m)^2 + (\gamma m)^2 = m^2$.
We let $\varepsilon=0.001$ and assume from now on that $m\geq 3257$.
In particular $\beta m, \gamma m, m\geq 1880$ so that we can apply \cref{lem:explicit_value_x0} in the following calculations. 
Given these parameters our choice of a feasible solution $Y = (Y_{k,\ell}) \in \R^{m \times m}$ is given by


\begin{align}
Y_{k,\ell} &=
\begin{cases}
& \phantom{ \qquad } \left(0 \leq k \leq \beta m \textrm{ and } \gamma m \leq  \ell \leq m \right) \textrm{ or} \\
C \, \varphi(k)\varphi(\ell) & , \textrm{ if \ } \left(\beta m\leq k \leq \gamma m \textrm{ and } \beta m\leq \ell \leq \gamma m\right) \textrm{ or} \\
& \phantom{ \qquad } \left(\gamma m\leq k \leq m \textrm{ and } 0\leq \ell\leq \beta m\right) \\
0 & ,\textrm{ otherwise}
\end{cases}
\label{eq:choosing_feasible_solution}
\end{align}
with a positive constant~$C$ that will be determined later on.
Note that the non-zero entries of~$Y$ correspond to the three gray rectangles from \cref{fig:shape-solutions-Y}~(b).

We now need to check that this is in fact a feasible solution, that is, it satisfies all the constraints in \cref{lp:dual_version}.
The matrix~$Y$ has by definition only non-negative entries, so we only need to consider the non-homogeneous constraints.
We start by splitting these constraints into three cases, depending on which interval $k$ belongs to:

First, let~$k$ be such that $0\leq k\leq \beta m$.
Then $Y_{k,\ell}\neq 0$ only for $\gamma m\leq \ell\leq m$.
Similarly, $Y_{\ell, k}\neq 0$ only for $\gamma m\leq \ell\leq m$ since $Y$ is symmetric by definition.
In view of \cref{lem:explicit_value_x0} and recalling $\varepsilon = 0.001$, we thus obtain
\begin{align*}
    \sum_{\ell=1}^m (Y_{k,\ell} + Y_{\ell, k}) 
        &= \sum_{\gamma m\leq \ell \leq m} Y_{k,\ell} + \sum_{\gamma m\leq \ell\leq m} Y_{\ell, k} \\
        &= 2C\varphi(k) \sum_{\gamma m\leq \ell\leq m} \varphi(\ell) \\
        &=2C\varphi(k) \left(\sum_{\ell\leq m} \varphi(\ell)-\sum_{\ell < \gamma m} \varphi(\ell)\right) \\
        &\geq 2C\varphi(k) \left((1-\varepsilon)\frac{3}{\pi^2}m^2 - (1+\varepsilon)\frac{3}{\pi^2}(\gamma m)^2\right) \\
        &= \frac{6}{\pi^2}Cm^2\varphi(k)\left((1-\varepsilon)-(1+\varepsilon)\gamma^2\right) \\
        &= \frac{6}{\pi^2}Cm^2\varphi(k) \left((1-\varepsilon)-(1+\varepsilon)\frac23\right) \\
        &\overset{!}{\geq }\varphi(k),
\end{align*}
so in order for this last inequality to hold we need that
\begin{align}
    C&\geq \frac{\pi^2}{6}\frac{1}{m^2}\left(\frac13-\frac53\varepsilon\right)^{-1}. \label{eq:lower_bound_C}
\end{align}
The cases $\beta m\leq k \leq \gamma m$ and $\gamma m\leq k\leq m$ can be done analogously, leading to the additional conditions
%
%
\[
C \geq \frac{\pi^2}{6}\frac{1}{m^2}\left(\frac13-\varepsilon\right)^{-1} \quad \textrm{ and } \quad
C \geq \frac{\pi^2}{6}\frac{1}{m^2}\left(\frac13-\frac13\varepsilon\right)^{-1} \,,
\]
respectively.
Clearly the strongest condition on~$C$ is made by inequality~\eqref{eq:lower_bound_C}. 

Now we consider the objective function.
We split the sum into three parts and use \cref{lem:explicit_value_x0} again.
Note that the sums run over integer indices in the given ranges, and that $\beta m$ and $\gamma m$ cannot be integers, for any~$m$.
\begin{align*}
    z_m^*
        &= \sum_{k,\ell=1}^m \frac{2}{k\ell} Y_{k,\ell} \\
        &= 2C\left(
            \sum_{\substack{0\leq k\leq \beta m \\ \gamma m\leq \ell \leq m}} \frac{\varphi(k)}{k}\frac{\varphi(\ell)}{\ell} 
            +\sum_{\substack{\beta m\leq k\leq \gamma m\\\beta m\leq \ell\leq \gamma m}} \frac{\varphi(k)}{k}\frac{\varphi(\ell)}{\ell}  
            +\sum_{\substack{\gamma m\leq k\leq m\\0\leq \ell\leq \beta m}} \frac{\varphi(k)}{k}\frac{\varphi(\ell)}{\ell}
            \right) \\
        &= 2C\left(
            2 \sum_{\substack{0\leq k\leq \beta m \\ \gamma m\leq \ell \leq m}} \frac{\varphi(k)}{k}\frac{\varphi(\ell)}{\ell} 
            + \left[\sum_{\beta m\leq k\leq \gamma m} \frac{\varphi(k)}{k}\right]^2
            \right) \\
        &= 2C\left(
            2 \left[\sum_{k\leq \beta m} \frac{\varphi(k)}{k}\right]\left[\sum_{\ell\leq m}\frac{\varphi(\ell)}{\ell}-\sum_{\ell \leq \gamma m}\frac{\varphi(\ell)}{\ell}\right]
            + \left[\sum_{k\leq \gamma m} \frac{\varphi(k)}{k} - \sum_{k \leq \beta m} \frac{\varphi(k)}{k}\right]^2
            \right) \\
        &\leq 2C\left(
            2\left[(1+\varepsilon)\frac{6}{\pi^2}\beta m\right] \left[(1+\varepsilon)\frac{6}{\pi^2}m - (1-\varepsilon)\frac{6}{\pi^2}\gamma m\right] \right. \\
        &\phantom{\leq 2C 2 \, } \left.    + \left[(1+\varepsilon)\frac{6}{\pi^2}\gamma m - (1-\varepsilon) \frac{6}{\pi^2}\beta m\right]^2
            \right) \\
        &= 2\left(\frac{6}{\pi^2}\right)^2 m^2 C\left(
            2\left[(1+\varepsilon)\beta\right] \left[(1+\varepsilon)-(1-\varepsilon)\gamma \right]
            + \left[(1+\varepsilon)\gamma -(1-\varepsilon)\beta\right]^2
            \right) \,.
\end{align*}
If we denote
\begin{equation*}
    f(\varepsilon) = 2\left[(1+\varepsilon)\beta\right] \left[(1+\varepsilon)-(1-\varepsilon)\gamma \right]
    + \left[(1+\varepsilon)\gamma -(1-\varepsilon)\beta\right]^2 \,,
\end{equation*}
then the desired condition $z_m^* \leq w$ will be satisfied as long as
\begin{equation}
    C \leq \frac w 2\left(\frac{\pi^2}{6}\right)^2 \frac{1}{m^2} \frac{1}{f(\varepsilon)} \,. \label{eq:upper_bound_C}
\end{equation}

We now obtain the following result:

\begin{lemma}
\label{lem:z_m_less_than_c_for_big_m}
Let $z_m$ be the optimal value of \cref{lp:dual_version}.
For every $m\geq 3257$, we have $z_m \leq 0.999$.
\end{lemma}

\begin{proof}
    For $\varepsilon=0.001$ and $w=0.999$ the inequalities \cref{eq:lower_bound_C,eq:upper_bound_C} imply
    \begin{equation*}
        \frac{4.9596}{m^2} \leq C \leq \frac{4.9678}{m^2} \,.
    \end{equation*}
    Hence, choosing $C=4.96/m^2$ we obtain that the solution given in \cref{eq:choosing_feasible_solution} is indeed feasible and its objective value is upper bounded by~$w$.
    Recall, that we chose $m\geq 3257$ in order to have $\beta m, \gamma m, m\geq 1880$, so that we could apply \cref{lem:explicit_value_x0} in the calculations above.
\end{proof}

\section{Closing the gap and proof of \texorpdfstring{\cref{thm:upper_bound_on_g_Delta2}}{Theorem~\ref{thm:upper_bound_on_g_Delta2}}}
\label{sect:closing-gap}

Here we show that we also have $z_m \leq 0.999$ in the range $4\leq m\leq 3257$, which is the final ingredient for the proof of \cref{thm:upper_bound_on_g_Delta2}.
To this end, we employ a computer to solve the relevant linear programs.
We use GLPK from GNU~\cite{gnuGNULinearProgramming} implemented in \texttt{sagemath}~\cite{thesagedevelopersSageMathSageMathematics2023}. 

While for small values of~$m$ the LP \cref{lp:primal_version} or its dual~\cref{lp:dual_version} can be easily solved, this is no longer the case for bigger~$m$.
Note that in the primal problem we have $m$ variables and $m^2$ constraints while in the dual problem we have $m^2$ variables and $m$ constraints, which makes the linear program too large for a computer to quickly solve it.
Hence for larger values of $m$ we will only give an upper bound for the optimal value~$z_m$.

Just as in \cref{sect:upper_bound_linear_program}, our approach is based on the observation that the optimal solutions~$Y$ of \cref{lp:dual_version}, for small values of $m$, satisfy $Y_{k,\ell}\neq 0$ only for index pairs $(k,\ell)$ with $k^2+\ell^2\approx m^2$.
Hence, we hope that we can still get a good enough approximation $z_m \leq \hat z_m \leq 0.999$ by setting $Y_{k,\ell}=0$ for other indices $(k,\ell)$, which is effectively achieved by not including those variables in the linear program.
More precisely, we only include variables~$Y_{k,\ell}$ for the pairs $(k,\ell)$ in the following set:
\begin{equation*}
    S = \left\{(k,\ell)\in \{1,\ldots,m\}^2: 1\leq k\leq \frac{1}{\sqrt2}m, \sqrt{(1-\varepsilon)m^2-k^2}\leq \ell\leq \sqrt{(1+\varepsilon)m^2-k^2}\right\} 
\end{equation*}
for some small $\varepsilon>0$.  The approximate linear program~\cref{lp:approximate_dual_version} is given by
\begin{align}
    \sum_{(k,\ell)\in S} \frac{2}{k\ell} Y_{k,\ell}&\to \min \notag \\*
    \sum_{(k,\ell)\in S \text{ or } (\ell,k)\in S} (Y_{k,\ell} + Y_{\ell,k}) &\geq \varphi(k), \qquad k=1,\ldots, m,\tag{$\hat D_m$}\label{lp:approximate_dual_version}\\*
    Y&\geq 0 \,. \notag 
\end{align}
The choice $\varepsilon=1.85/m$ works well for our purposes, leading to roughly $2m$ variables, making the LP \eqref{lp:approximate_dual_version} much easier to solve than \eqref{lp:dual_version}.
If the optimal value $\hat z_m$ of~\eqref{lp:approximate_dual_version} is bigger than~$0.999$, we can simply increase $\varepsilon$ by a factor of~$2$ and try again.
In practice, this was only needed for a few small values~$m$.

However, it would still take a considerable amount of time to solve all approximate LPs~\eqref{lp:approximate_dual_version} for $4\leq m\leq 3257$, which is why we employ the following lemma:

\begin{lemma}\label{lem:bound_zm+1_by_zm}
Let $z_m$ be the optimal value of~\eqref{lp:dual_version}.
Then
    \begin{equation*}
        z_{m+1} \leq z_m + \frac{\varphi(m+1)}{(m+1)^2} \,.
    \end{equation*}
\end{lemma}

\begin{proof}
    Let $Y\in \R^{m\times m}$ be an optimal solution of~\eqref{lp:dual_version}.
    We construct a feasible solution~$\hat Y$ for $(D_{m+1})$ as follows:
    \begin{equation*}
        \hat{Y} = \begin{pmatrix}
            Y & 0 \\
            0 & \frac{\varphi(m+1)}{2}
        \end{pmatrix} \in \R^{(m+1)\times(m+1)} \,.
    \end{equation*}
    As $Y$ is feasible for~\eqref{lp:dual_version}, we see that~$\hat Y$ satisfies the first $m$ constraints of~$(D_{m+1})$.
    The $(m+1)$-st constraint reads 
    \begin{equation*}
        \sum_{\ell=1}^{m+1} \left(\hat Y_{m+1,\ell} + \hat Y_{\ell,m+1}\right) = 2 \hat Y_{m+1,m+1} = \varphi(m+1) \geq \varphi(m+1)
    \end{equation*}
    and is also satisfied.
    Hence $\hat Y$ is feasible for $(D_{m+1})$. The objective value of $\hat Y$ equals
    \begin{equation*}
        \sum_{k,\ell=1}^{m+1} \frac{2}{k\ell} \hat{Y}_{k,\ell} = \sum_{k,\ell=1}^m \frac{2}{k\ell} Y_{k,\ell} + \frac{2}{(m+1)^2}\hat Y_{m+1,m+1} = z_m + \frac{\varphi(m+1)}{(m+1)^2} \,.
    \end{equation*}
    Therefore, the optimal value $z_{m+1}$ of $(D_{m+1})$ satisfies
    \begin{equation*}
        z_{m+1} \leq z_m + \frac{\varphi(m+1)}{(m+1)^2} \,,
    \end{equation*}
    as claimed.
\end{proof}

We remark that, in view of~\cref{eq:asymptotic_phi_div_k^2}, the series
\begin{equation*}
    \sum_{m=1}^\infty \frac{\varphi(m)}{m^2}
\end{equation*}
diverges.
Hence, it would not be enough to show that $z_m \leq 0.999$ for sufficiently large values of~$m$ until the tail of the series converges to a small enough value and then conclude $z_m \leq \bar w < 1$, for some other constant~$\bar w$, say $\bar w = 0.9999$.

The desired speed up in our computations now comes as follows:
Denote by $\hat z_m$ the optimal value of $(\hat D_m)$.
Then, by \cref{lem:bound_zm+1_by_zm}, we have
\begin{equation*}
    z_{m+1} \leq z_m + \frac{\varphi(m+1)}{(m+1)^2}\leq \hat z_m + \frac{\varphi(m+1)}{(m+1)^2} \,,
\end{equation*}
and we do not need to solve $(\hat D_{m+1})$ if $\hat z_m + \frac{\varphi(m+1)}{(m+1)^2} \leq 0.999$.
Using these observations, it takes about a minute of computer time (see \url{https://github.com/BKriepke/Generic-Delta-modular-with-two-rows}) to check that indeed
\[
z_m \leq 0.999 \quad\textrm{ for }\quad 4 \leq m \leq 3257 \,.
\]
In summary, together with \cref{lem:z_m_less_than_c_for_big_m}, we obtain

\begin{lemma}
\label{lem:z_m_less_than_c_for_all_m}
We have $z_m \leq 0.999$, for every integer $m \geq 4$.
\end{lemma}

Finally, we are in the position to give a proof of our main result \cref{thm:upper_bound_on_g_Delta2}.

\begin{proof}[Proof of \cref{thm:upper_bound_on_g_Delta2}]
%
    Let $\Delta\geq 10^8$ and let $A$ be a generic $\Delta$-modular matrix with two rows.
    Let~$M$ be the set of integers $m\leq \sqrt{\tfrac{\pi}{2}}\sqrt\Delta$ such that there exists a matrix $M(a,b)$ of type $m$ with at least as many columns as~$A$.
    By \cref{prop:omnipresence_matrices_M_a_b} the set $M$ is nonzero. Assume first that $M\setminus\{1,2,3\}$ is nonempty, let $m\in M\setminus\{1,2,3\}$ be arbitrary and let $M(a,b)$ be a matrix of type $m$ with at least as many columns as~$A$.
    By \cref{lem:z_m_less_than_c_for_all_m} we have $z_m \leq 0.999$.
    From the concrete bound~\eqref{eq:implied_constant_key_proposition} in the proof of \cref{prop:upper_bound_number_columns_version_2}, we get
    \begin{align*}
        \abs{A}&\leq \abs{M(a,b)} \\
            &\leq 0.999\Delta + 1+ \frac{\sqrt{\frac\pi2\Delta}}{\zeta(2)}\left(\log \sqrt{\frac\pi2\Delta} + 2\gamma - 1 - \frac{2\zeta'(2)}{\zeta(2)}\right) + 15\left(\frac\pi2\Delta\right)^{1/4}\log \sqrt{\frac\pi2\Delta} \\
            &\leq 0.999\Delta + 0.001\Delta = \Delta \,,
    \end{align*}
    where the last inequality holds because $\Delta\geq 10^8$.
    
    Otherwise we have $\emptyset\neq M\subseteq \{1,2,3\}$ and using \cref{cor:precise_upper_bound_for_m_1_2_3} we obtain
    \begin{equation*}
        \abs{A} \leq \abs{M(a,b)} \leq \tilde\g(\Delta) =
            \begin{cases}
                \Delta + 4 & ,\textrm{ if }\Delta \equiv 2 \bmod 6 \\
                \Delta + 3 & ,\textrm{ if }\Delta \equiv 1,3,5 \bmod 6 \\
                \Delta + 2 & ,\textrm{ if }\Delta \equiv 0,4 \bmod 6 \,.
            \end{cases}
    \end{equation*}
    As mentioned in the introduction, the lower bound $\g(\Delta,2) \geq \tilde\g(\Delta)$ was shown in~\cite{kriepkeSizeIntegerPrograms2025}, see also \cref{ex:familiesF1_to_F3}.
\end{proof}

\section{Improvements and limitations of our approach}
\label{sect:limitations}

As mentioned in the introduction, the constant $\Delta_0=10^8$ in \cref{thm:upper_bound_on_g_Delta2} is not optimal.
In this last section, we discuss possible improvements in our proof that would reduce the value of~$\Delta_0$.
At the same time, we point to the limitations of the approach, showing that without significantly new ideas, it is not possible to reduce $\Delta_0$ far enough to be able to compute $\g(\Delta,2)$ for all $\Delta\leq \Delta_0$ with a computer in a reasonable timeframe.

The proof of \cref{thm:upper_bound_on_g_Delta2} shows that~$\Delta_0$ is essentially determined by the inequality
\begin{equation*}
    \frac{1}{\zeta(2)}\cdot c_1\sqrt\Delta\log\left(c_1\sqrt\Delta\right)\leq (1-w)\Delta\,,
\end{equation*}
where we ignore the precise error term of \cref{prop:upper_bound_number_columns_version_2} for the moment in favour of simplicity. Here, $c_1=\sqrt{\pi/2}$ is the constant for which $m\leq c_1\sqrt\Delta$ holds in \cref{prop:omnipresence_matrices_M_a_b} and~$w$ is the value for which we show $z_m \leq w$, for $m\geq 4$. In particular, any improvement to the error term of \cref{prop:upper_bound_number_columns_version_2} or the values $c_1$ and $w$ will also improve~$\Delta_0$.
As we do not expect significant improvements on~$c_1$ or said error term, we only discuss possibilities to reduce~$w$.

We have shown $z_m \leq 0.999$, for $m\geq 4$ in two steps.
A theoretical argument covered all $m\geq m_0=3257$, and we used a computer for the range $4 \leq m \leq m_0$. 

For the first step, we presented a feasible solution~$Y$ to~\cref{lp:dual_version} based on approximating the quarter circle shown in \cref{fig:shape-solutions-Y} by $p=3$ rectangles. This gave us two conditions for the constant $C$, one of the form $C\geq g(\beta,\gamma,\varepsilon)$ to satisfy the constraints and one of the form $C\leq h(\beta,\gamma,\varepsilon,w)$ to satisfy $z_m^*\leq w$, for some functions $g,h$. These conditions could then be rearranged to a form $w\geq k(\beta,\gamma,\varepsilon)$ (although we have not done so). In our case we got a function monotonically increasing in $\varepsilon>0$, meaning a smaller $\varepsilon$ gives a smaller~$w$. However, decreasing $\varepsilon$ comes at the cost of increasing the value $x_0(\varepsilon)$ in \cref{lem:explicit_value_x0}. This implies an increase of~$m_0$, because it needs to satisfy $\beta m_0 \geq x_0(\varepsilon)$. 

We may improve the approximation of the quarter circle using $p \geq 4$ rectangles and a sequence $0<\beta_1<\beta_2<\ldots<\beta_{p-1}<1$ instead of $0<\sqrt{1/3}<\sqrt{2/3}<1$, for example $\beta_i = \tan\left(\frac ip\frac{\pi}{2}\right)$. This would presumably allow for a better bound~$w$. However, it again comes at the cost of increasing~$m_0$, due to the constraint $\beta_1 m_0 \geq x_0(\varepsilon)$.

For the second step, we used a computer to solve the approximate linear programs~\cref{lp:approximate_dual_version} to show $z_m \leq \hat z_m \leq w$, for all $4\leq m\leq m_0$, with the help of \cref{lem:bound_zm+1_by_zm} for a speed up.
Of course, lowering~$w$ or increasing~$m_0$ leads to a higher computational effort.
More importantly, there will be values of~$m \geq 4$ for which the inequality $z_m \leq w$ does not hold. A concrete example is $z_5=119/120 = 0.991\overline6$ and any $w<119/120$. Any such value of $m$ needs to be dealt with individually. This could come in the form of a proof similar to the proofs of \cref{claim:matrices_type_2,claim:matrices_type_3} for $m=2,3$, however it is not obvious how to generalize these proofs without ending up with too many cases to check. Another approach is to compute the error term of \cref{prop:upper_bound_number_columns_version_2} explicitly. As an example, for $m=5$ we have $\abs{M(a,b)}\leq \frac{119}{120}\Delta + 7$. The inequality $\frac{119}{120}\Delta + 7\leq \Delta+2$ then gives that for $\Delta\geq 600$ no matrix $M(a,b)$ of type $5$ could give a generic $\Delta$-modular matrix with more columns than a matrix of type $1,2$ or $3$, meaning we only have to check $\Delta\leq 600$. While for $m=5$ this is doable, for bigger $m$ it quickly becomes unfeasible.

Furthermore, the value of $w$ can likely not be reduced further than $w\approx 0.955$. Based on our numerical data it seems that $z_m$ converges to a value near $0.955$.
Even if we could prove that indeed $z_m \leq 0.955$, for $m$ large enough, then using $w=0.955$ in the relevant inequality of the proof of \cref{thm:upper_bound_on_g_Delta2} gives $\Delta_0\approx 65000$. While this is certainly a huge improvement over $\Delta_0=10^8$, it is still not feasible to compute $\g(\Delta,2)$ for all $\Delta\leq \Delta_0$ and thus determine $\g(\Delta,2)$ completely.

\subsection*{Acknowledgments}

We thank Zach Walsh for valuable comments on an earlier version and for pointing out the implications of our results in matroid theory as described in \cref{cor:matroids}.

\bibliographystyle{amsplain}
\bibliography{refs}

\end{document}